\documentclass[reqno,12pt]{amsart}

\usepackage{amssymb,amsmath,amsthm} 

\newcommand{\Z}{\mathbb{Z}}

\newcommand{\legendre}[2]{\genfrac{(}{)}{}{}{#1}{#2}}

\newtheorem{theorem}{Theorem}

\newtheorem{lemma}{Lemma}

\author{Jing-Jing Huang}
\address{
JJH: Department of Mathematics and Statistics, University of Nevada, Reno,
1664 N. Virginia St., Reno, NV 89557}
\email{jingjingh@unr.edu}
\author{Huixi Li}
\address{
HL: Department of Mathematics and Statistics, University of Nevada, Reno,
1664 N. Virginia St., Reno, NV 89557}
\email{huixil@unr.edu}

\dedicatory{}
\thanks{The first named author is supported by the UNR VPRI startup grant 1201-121-2479}
\keywords{lattice points, character sums, Gauss sums, rational points near the parabola}
\subjclass[2010]{Primary 11J25, Secondary 11P21}
\begin{document}
\title{On two lattice points problems about the parabola}
\begin{abstract}
    We obtain asymptotic formulae with optimal error terms for the number of lattice points under and near a dilation of the standard parabola, the former improving upon an old result of Popov. These results can be regarded as achieving the square root cancellation in the context of the parabola, whereas  its analogues are wide open conjectures for the circle and the hyperbola. We also obtain essentially sharp upper bounds for the  latter lattice points problem associated with the parabola. Our proofs utilize techniques in Fourier analysis, quadratic Gauss sums and character sums. 
\end{abstract}

\maketitle

\section{Introduction}

The Gauss circle problem is one of the celebrated open questions in number theory. It asks for the best possible error term when approximating the number of lattice points inside a dilating circle centered at the origin with its area. More precisely, it is conjectured that for $a\ge1$
\begin{equation}\label{e1.1}
\sum_{0\le x\le a}\left\lfloor\sqrt{a^2-x^2}\right\rfloor=\frac{\pi}4a^2+O_\varepsilon\left(a^{\frac12+\varepsilon}\right).
\end{equation}

A concomitant conjecture for the hyperbola states that
\begin{equation}\label{e1.2}
    \sum_{x\le a^2}\left\lfloor\frac{a^2}{x}\right\rfloor=2a^2\log a+(2\gamma-1)a^2+O_\varepsilon\left(a^{\frac12+\varepsilon}\right),
\end{equation}
which, after normalization, is equivalent to the Dirichlet divisor problem
$$
\sum_{n\le a}d(n)=a\log a+(2\gamma-1)a+O_\varepsilon(a^{\frac14+\varepsilon}).
$$
We note that the expressions on the left sides of \eqref{e1.1} and \eqref{e1.2} represent the number of lattice points in the first quadrant that are under the dilation of the circle $x^2+y^2=1$ and the hyperbola $xy=1$ via the transformation $(x,y)\to\left(\frac{x}a,\frac{y}a\right)$, which are $x^2+y^2=a^2$ and $xy=a^2$, respectively. 

In spite of the very rich literature on the above problems, the conjectures remain well out of reach with the current technology. We only mention that recently Bourgain and Watt \cite{BourgainWatt} have obtained the best known error $O(a^\frac{517}{824})$ for both the conjectures \eqref{e1.1} and \eqref{e1.2}, improving on earlier bound  $O(a^\frac{131}{208})$ due to Huxley \cite{Hux1, Hux2}.

As there are only three types of conics, the ellipse, parabola and hyperbola, one 
can just as easily consider the analogous problem for the parabola. 
Indeed, Popov \cite{Popov} obtained in 1975 the first result in this regard. Note that the dilation of the standard parabola $y=x^2$ under the transformation $(x,y)\to\left(\frac{x}a,\frac{y}a\right)$ is $y=\frac{x^2}{a}$. Popov's result states that for a large real number $a$ and a positive integer $b\ll a$, 
\begin{equation}\label{e1.3}
\sum_{ x \leq b} \left\lfloor \frac{x^2}{a} \right\rfloor
=
\sum_{x \leq b}
\left( \frac{x^2}{a} - \frac{1}{2} \right)
+ 
O \left( a^{\frac{1}{2} + \frac{c}{\log \log a}} \right), 
\end{equation}
where $c$ is a positive constant independent of $a$. Clearly the number of lattice points $(x,y)\in\mathbb{Z}^2$ in the region $0 < x \leq b$, $0 < y \leq \frac{x^2}{a}$ is  equal to $\displaystyle \sum_{ x \leq b} \left\lfloor \frac{x^2}{a} \right\rfloor$.

As Popov points out, the exponent $1/2$ in the error term of \eqref{e1.3} is best possible in general. Nevertheless, our Theorem~\ref{t1} shows that \eqref{e1.3} is subject to further improvement when $b \ll a$ and $a$ is an integer. 

\begin{theorem}\label{t1} For any positive integers $a, b$ with $a\ge3$,
we have 
\[
\sum_{ x \leq b} \left\lfloor \frac{x^2}{a} \right\rfloor
=
\sum_{x \leq b}
\left( \frac{x^2}{a} - \frac{1}{2} \right)
+ 
O \left( \sqrt{a} \log a+ba^{-\frac12+\frac{2+o(1)}{\sqrt{\log a}\log\log a} }\right).
\]
\end{theorem}

Interestingly, when $a=b$ the minimum order of the error is studied by Chamizo and Paster. They prove in \cite[Proposition 5.2]{ChamizoPastor} that
\[ 
\sum_{ x \leq a} 
\left\lfloor \frac{x^2}{a} \right\rfloor
-
\sum_{x \leq a}
\left( 
\frac{x^2}{a}-\frac12\right) = \Omega\left(a^{\frac{1}{2}} \exp\left(\frac{(\sqrt{2} - \varepsilon) (\log a)^{\frac{1}{2}}} {\log \log a}\right)
\right). \]
This shows that Theorem \ref{t1} is surprisingly close to the true order of the error term.

A problem of similar flavor is to estimate the number of lattice points close to the parabola. Let
$$
A(a, b, \delta) 
= \sum_{\substack{x \leq b \\ \left\|\frac{x^2}{a}\right\| \leq \delta}}1.$$
The function $A(a,b,\delta)$ naturally counts the number of rational points $(\frac{x}a,\frac{y}a)$ with $x\le b$ that lie close to $\mathcal{P}:y=x^2, x\in(0,1]$, or equivalently the number of lattice points $(x,y)\in\Z^2$ with $x\le b$ close to $a\mathcal{P}: y=\frac{x^2}a, x\in[1,a]$. Furthermore, if $\delta<1/a$, then all such points counted by $A(a,b,\delta)$ must be forced to lie on the curve, namely $A(a,b,\delta)=A(a,b,0)$ for $\delta\in(0,1/a)$. We obtain the following essentially optimal estimate of $A(a,b,\delta)$.
\begin{theorem}\label{t2} For any positive integers $a, b$ with $a\ge 3$ and $\delta\in(0,\frac12)$, we have
$$A(a,b,\delta)=2\delta b+O\left(  \sqrt{a}\log a+ba^{-\frac12+\frac{2+o(1)}{\sqrt{\log a}\log\log a} }\right).$$
\end{theorem}

When $b\asymp a$, Theorem~\ref{t1} and Theorem~\ref{t2} can be regarded as achieving the square root cancellation for the two lattice points problems on the parabola,  which are best possible in general. Nonetheless, if we are only concerned with upper bounds, better results are available, which, in some cases, can even beat the square root cancellation. 
Denote \[
A(a, \delta) = A(a, a, \delta) 
= \sum_{\substack{x \leq a \\ \left\| \frac{x^2}{a} \right\| \le \delta }} 1. 
\]

The first author recently proved the following result\footnote{The result there has $r^{1+\varepsilon}$ instead of $\sigma(r)$ as quoted here. But it is straightforward to deduce the latter from the proof of \cite{Huang}.}. 
\begin{theorem}[{\cite[Theorem 4]{Huang}}]\label{t0}
Let $a$ be a positive integer, $r$ be the largest integer such that $r^2|a$ and $\delta\in(0,\frac12)$. Then for any $\epsilon > 0$, we have 
\[
A(a,\delta) \ll \delta a + \sigma(r) + \delta^{\frac{1}{2}} a^{\frac{11}{16} + \epsilon}, 
\]
where the implied constant only depends on $\epsilon$. 
\end{theorem}
Here $\sigma(r)= \sum_{d|r}d$. It is well known that $\sigma(r)\ll r\log\log 3r$, see \cite[\S I.5.5 Theorem 5]{Ten}.  Theorem~\ref{t0} has been used to solve some questions regarding metric Diophantine approximation on the parabola \cite{Huang}.   

Assuming lattice points are randomly distributed in the $\delta$ neighborhood of the dilation $a\mathcal{P}$, we expect roughly $2\delta a$ such points. 
Also, let $a=r^2s$ with $s$ squarefree, then we observe that there are exactly $r$ lattice points lying on $a\mathcal{P}$, namely $(rsl, sl^2)$, $l=1, 2, \ldots, r$. So even with $\delta\to0$, the upper bound in Theorem~\ref{t0} cannot be less than $r$ since those $r$ lattice points on $a\mathcal{P}$ are always counted by $A(a,\delta)$ for any $\delta>0$.   Naturally, the second term $\sigma(r)$ accounts for this phenomenon, and indeed it is not very far from $r$ as noted above. Therefore neither of these two terms can be dispensable, nor can they be improved much.  However, the third term is most likely not optimal and should be susceptible to further improvement. In fact, assuming the Generalized Lindel\"{o}f Hypothesis \cite[\S12.4]{IK}, Theorem \ref{t0} can be improved to
$$
A(a,\delta) \ll \delta a + \sigma(r) + \delta^{\frac{1}{2}} a^{\frac{1}{2} + \epsilon}.
$$

It is easily seen that the term $\delta^{\frac{1}{2}}a^{\frac{11}{16} + \epsilon}$ is less than the heuristic main term $\delta a$ only when  $\delta\ge a^{-\frac58+\varepsilon}$, for some $\varepsilon>0$. We are able to improve the bound in Theorem~\ref{t0} when $\delta\ll a^{-\frac58}$. 
In fact, our Theorem~\ref{t3} below is sharp up to an $a^\varepsilon$ loss in the main term.

\begin{theorem}\label{t3} Let $a$ be a positive integer, $r$ be the largest integer such that $r^2|a$ and $\delta\in(0,\frac12)$. Then 
$$A(a,\delta)\ll \delta a(\log a) d(a) + \sigma(r).$$
\end{theorem}
Here $d(a)= \sum_{d|a}1$ is the divisor function.
It is well known that $d(a)\le a^{(1+o(1))\frac{\log2}{\log\log a}}\ll_\varepsilon a^{\varepsilon}$, see \cite[\S I.5.2 Theorem 2]{Ten}.

It is worth noting that the core of the proofs of Theorem~\ref{t0} and Theorem~\ref{t3} utilize  estimates of character sums, after applying some elementary Fourier analysis and classical results about Gauss sums. The difference is that Burgess's bound \cite{Burgess} is used for Theorem~\ref{t0} while the P\'{o}lya-Vinogradov inequality (Lemma~\ref{l3}) is used for Theorem~\ref{t3}, when estimating such character sums. Since the latter can be improved under the Generalized Riemann Hypothesis, we also have the following theorem. 

\begin{theorem}\label{t4}
Under the assumption of the Generalized Riemann Hypothesis and the same conditions of Theorem~\ref{t3}, we have
$$
A(a,\delta)\ll \delta a(\log \log 3a)d(a) + \sigma(r).
$$
\end{theorem}

We only remark in passing that the problem of obtaining upper and lower bounds for the number of lattice/rational points near a manifold has become a very active area of research, see \cite{B, BDV, Beresnevich, H1, H2, H3, H4, VV} and the references therein for the background and recent progress. 

Throughout the paper, we will use the notation $e(x)=e^{2\pi i x}$, $\|x\|=\min_{n\in\mathbb{Z}}|x-n|$, $\{x\}=x-\lfloor x\rfloor$ and Vinogradov's symbol $f(x)\ll g(x)$ and Landau's symbol $f(x)=O(g(x))$ to mean there exists a constant $C$ such that $|f(x)|\le Cg(x)$. The symbol $f(x) = o(g(x))$ means $\lim_{x \to \infty} \frac{f(x)}{g(x)} = 0$, and $f(x) = \Omega(g(x))$ means the negation of $f(x) = o(g(x))$. 

\section{The proof of Theorem~\ref{t1}}\label{s2}
We start by observing that, in view of the orthogonality of additive characters
\[ 
\frac{1}{a} \sum_{h \bmod a} e \left( \frac{h d}{a} \right)
= \begin{cases} 
     1, &\textrm{if} \;\; a \mid d, \\
     0, & \text{otherwise, } 
   \end{cases}
\]
we have
\begin{align*}
\sum_{x \leq b}
\left\{ \frac{x^2}{a} \right\}
=& 
\sum_{x \le b}
\sum_{\substack{0 \leq j \leq a - 1\\x^2\equiv j\bmod{a}}}
\frac{j}{a} \\
=& 
\sum_{0 \leq j \leq a - 1}
\sum_{x \le b}
\sum_{h \bmod{a}}
\frac{1}{a}
e\left( h \frac{x^2 - j}{a} \right)
\frac{j}{a} \\
=& 
\frac{1}{a^2}
\sum_{h \bmod{a}}
\sum_{0 \leq j \leq a - 1}
j e\left( - \frac{h j}{a} \right)
\sum_{x \leq b}
e\left( h \frac{x^2}{a} \right). 
\end{align*}

It follows by an elementary calculation that 
\[  \sum_{0 \leq j \leq a - 1} j e\left( - \frac{h j}{a} \right)
= \begin{cases} 
      \frac{-a}{1 - e\left( - \frac{h}{a} \right) }, & h \neq 0, \\
      \frac{(a - 1)a}{2}, & h = 0. 
   \end{cases}
\]
Therefore, we have 
\begin{equation}\label{e2.1}
\sum_{x \leq b}
\left\{ \frac{x^2}{a} \right\}
= \frac{1}{a^2} \frac{(a - 1) a}{2} b
+ \frac{1}{a}
\sum_{1 \leq h \leq a - 1} 
\frac{-1}{1 - e\left( - \frac{h}{a} \right)} S(h, a, b), 
\end{equation}
where $S(h, a, b) = \displaystyle \sum_{x \leq b} e \left( \frac{h x^2}{a} \right)$ is an incomplete Gauss sum. 

Note that
\[
\left|\frac{1}{1 - e\left( - \frac{h}{a} \right)}\right|
=
\frac{1}{2\sin \left( \frac{h}{a} \pi \right) }\le \frac1{4\|\frac{h}a\|}.
\] 
Now we are poised to estimate the Gauss sum $S(h, a, b)$. To that end, we quote the following result of Korolev. 
\begin{lemma}[{\cite[Corollary, Page 53]{Korolev}}]
Let $a$, $b$ and $h$ be integers such that $1 \leq b \leq a$ and $(a, h) = 1$. Then
\[
\left| \sum_{x = 1}^b e\left(\frac{h x^2 }a\right) \right|
< 3.9071 \sqrt{a}. 
\]
\end{lemma}

It then follows immediately, by dissecting the range $[1,b]$ into blocks of length $a/(a,h)$ if necessary, that
\begin{align*}
S(h, a, b)=&\sum_{x \leq b} e \left(  \frac{h x^2 /(a,h)}{a/(a,h)} \right)\\
\ll& \left(\frac{b}{a/(a,h)}+1\right)\sqrt{\frac{a}{(a,h)}}\\
\ll&\frac{b}{\sqrt{a}}\sqrt{(a,h)}+\sqrt{\frac{a}{(a,h)}}.
\end{align*}
Hitherto it remains to estimate the sum
$$
\sum_{1 \leq h \leq a - 1} 
\frac{-1}{1 - e\left( - \frac{h}{a} \right)} S(h, a, b)\ll \sum_{1 \leq h \leq a - 1} 
\frac1{\|\frac{h}a\|} \left(\frac{b}{\sqrt{a}}\sqrt{(a,h)}+\sqrt{\frac{a}{(a,h)}}\right).
$$
We may split the latter sum into two sums $\displaystyle\sum_{1\le h<a/2}$ and $\displaystyle\sum_{a/2\le h<a}$, and will only treat the first case and note that the second is analogous. Thus
\begin{align}
  &   \sum_{h <a/2} 
\frac1{\|\frac{h}a\|} \left(\frac{b}{\sqrt{a}}\sqrt{(a,h)}+\sqrt{\frac{a}{(a,h)}}\right)\nonumber\\
=&\sum_{d|a}\sum_{\substack{k<\frac{a}{2d}\\(k,a/d)=1}} \frac{a}{kd} \left(\frac{b}{\sqrt{a}}\sqrt{d}+\sqrt{\frac{a}{d}}\right)\nonumber\\
\ll&a\log a\left(\frac{b}{\sqrt{a}}\sum_{d|a}\frac1{\sqrt{d}}+\sqrt{a}\right)\nonumber\\
\ll&a\left( \sqrt{a} \log a+ba^{-\frac12+\frac{2+o(1)}{\sqrt{\log a}\log\log a} }\right)\label{e2.2}, 
\end{align}
where in the last line the following bounds $$\sigma_{-\frac12}(a)=\sum_{d|a}\frac1{\sqrt{d}}\le \exp\left(\frac{(2+o(1))\sqrt{\log a}}{\log\log a}\right)$$
and
$$
\log a=\exp(\log\log a)\ll \exp\left(o\left(\frac{\sqrt{\log a}}{\log\log a}\right)\right)
$$
are used. For the former bound, see \cite[\S I.5.5 Theorem 5]{Ten}.

Therefore, we obtain from \eqref{e2.1} that
\[
\sum_{x \leq b}
\left( \frac{1}{2} - \left\{ \frac{x^2}{a} \right\} \right)=O\left( \sqrt{a} \log a+ba^{-\frac12+\frac{2+o(1)}{\sqrt{\log a}\log\log a} }\right),
\]
and Theorem~\ref{t1} follows immediately on noting that
\[
\sum_{x \leq b} \left\lfloor \frac{x^2}{a} \right\rfloor 
= 
\sum_{x \leq b} 
\left( \frac{x^2}{a} - \frac{1}{2} \right)
+
\sum_{x \leq b}
\left( \frac{1}{2} - \left\{ \frac{x^2}{a} \right\} \right).
\]

\section{The proof of Theorem~\ref{t2}}

Let  $J = \lfloor \delta a \rfloor$. 
Our goal is to count the number of integers $x \leq b$ such that $\left\|\frac{x^2}{a}\right\| \le \delta$. Note that $\left\|\frac{x^2}{a} \right\| \le \delta$ if and only if there exists an integer $k$, such that 
\[
k - \delta \le \frac{x^2}{a} \le k + \delta, 
\]
i.e.
\[
k a - \delta a \le x^2 \le k a + \delta a, 
\]
which happens if and only if $x^2 \equiv j \pmod{a}$ for some $|j| \leq J$. Therefore, by the orthogonality of additive characters, we have 
\begin{align*}
A(a, b, \delta) 
=& \sum_{\substack{x \leq b \\ \left\|\frac{x^2}{a}\right\| \leq \delta}}1 \\=& \sum_{|j| \leq J} \sum_{\substack{x \leq b\\x^2\equiv j\bmod a}} 1 \\
=& \sum_{|j| \leq J} \sum_{x \leq b} \frac{1}{a}
\sum_{h \bmod a} e\left( h \frac{x^2 - j}{a} \right) \\
=& \frac{1}{a}
\sum_{h \bmod{a}}
\sum_{|j| \leq J}
e\left( - \frac{h j}{a} \right)
\sum_{x \leq b}
e\left( \frac{h x^2}{a} \right).
\end{align*}
The term $h = 0$ contributes $(2 J + 1)\frac{b}a$ in the sum. When $h \neq 0$, since
\[
\sum_{|j| \leq J} e \left( - \frac{h j}{a} \right) \ll \left\|\frac{h}{a}\right\|^{-1}
\]
and 
\[
\sum_{x \leq b} e \left( \frac{h x^2}{a} \right) \ll \frac{b}{\sqrt{a}}\sqrt{(a,h)}+\sqrt{\frac{a}{(a,h)}}, 
\]
we have by \eqref{e2.2}
\begin{align*}
A(a, b, \delta) =& (2 J + 1)\frac{b}a+\sum_{1 \leq h \leq a - 1} 
\frac1{a\|\frac{h}a\|} \left(\frac{b}{\sqrt{a}}\sqrt{(a,h)}+\sqrt{\frac{a}{(a,h)}}\right)\\
=&2 \delta b + O\left( \sqrt{a} \log a+ba^{-\frac12+\frac{2+o(1)}{\sqrt{\log a}\log\log a} }\right). 
\end{align*}
This completes the proof of Theorem~\ref{t2}.

\section{The proofs of Theorem~\ref{t3} and Theorem~\ref{t4}}

To prove Theorem~\ref{t3}, we start by showing that
\[A(a,\delta)\ll \delta a(\log a)d(a) +\sigma(r), 
\]
where $r$ is the largest integer such that $r^2|a$. 

Let $H=\left\lfloor \frac{1}{2\delta}\right\rfloor$. Recall that the Fej\'{e}r kernel, defined as
$$
\mathcal{F}_H(t):=\sum_{h=-H}^H\frac{H-|h|}{H^2}e(ht),
$$
satisfies $\mathcal{F}_H(t)\ge \frac4{\pi^2}$ when $\|t\|\le \delta$ and $\mathcal{F}_H(t)\ge 0$ for all $t$.
Noting that
\[
A(a,\delta)
= \sum_{\substack{x \leq a \\ \left\| \frac{x^2}{a} \right\| \le \delta }} 1
= \sum_{\substack{x \leq a}} \mathbf{1}_{[0, \delta]} \left( \left\| \frac{x^2}{a} \right\| \right), \]
where $\mathbf{1}_{[0, \delta]}$ is the indicator function on $[0, \delta]$, 
we have
\[
A(a,\delta)
\ll
\sum_{x\le a} \mathcal{F}_H 
\left( \frac{x^2}{a}\right). 
\]
Therefore, we only need to estimate the sum $\displaystyle \sum_{x\le a} \mathcal{F}_H\left( \frac{x^2}{a}\right)$. 

First we single out the term with $h=0$ and obtain
\[
\sum_{x\le a} \mathcal{F}_H\left( \frac{x^2}{a}\right)
\ll \delta a+E(a), 
\]
where 
$$E(a)\ll \sum_{h=1}^H\frac{H-h}{H^2}\sum_{x=1}^ae\left(\frac{hx^2}a\right).$$
We denote the following complete Gauss sums by 
$$S(h,a):=\sum_{x=1}^ae\left(\frac{hx^2}a\right).$$ 
Then we have 
\begin{align*}
E(a)\ll &\sum_{h=1}^H\frac{H-h}{H^2}(h,a)S\left(\frac{h}{(h,a)},\frac{a}{(h,a)} \right)\\
\ll& \sum_{d|a}d\sum_{\substack{h\le H\\(h,a)=d}}\frac{H-h}{H^2}S\left(\frac{h}{d},\frac{a}{d}\right) \\
\ll& \sum_{d \mid a} d
\sum_{\substack{h_1\le H/d\\(h_1,a/d)=1}}\frac{H-h_1 d}{H^2}S\left(h_1, \frac{a}{d} \right).
\end{align*}
For a fixed $d \mid a$, let $a_1=a/d$.  To apply partial summation to the inner sum in the last line, we focus on the partial sum 
\[
S(N) = \sum_{\substack{h_1\le N \\(h_1,a_1)=1}}
S(h_1,a_1). 
\]
Compared to the incomplete Gauss sums $S(h,a,b)$ considered in \S \ref{s2},  the complete Gauss sums $S(h,a)$ are very well understood. Actually, we know their exact values.
\begin{lemma}[{\cite[\S3.5]{IK}}]\label{l2}
Suppose $(h,a)=1$. Then
\[
S(h,a)=
\left\{
\begin{array}{ll}
0,&\text{when } a\equiv2\pmod{4},\\
\varepsilon_a\legendre{h}{a}\sqrt{a},&\text{when } a \text{ is odd},\\
(1+i)\varepsilon_h^{-1}\legendre{a}{h}\sqrt{a},&\text{when } h \text{ is odd and }4|a, \\
\end{array}
\right.
\]
where 
\[
\varepsilon_m=
\left\{
\begin{array}{ll}
1,&\text{when }m\equiv1\pmod{4}, \\
i,&\text{when }m\equiv 3\pmod{4}, 
\end{array}
\right.
\]
and $\legendre{*}{*}$ is the Jacobi symbol. 
\end{lemma}
Note that $\legendre{*}{a}$ is a Dirichlet character modulo $a$ and $\legendre{a}{*}$ is a Dirichlet character of conductor $a'|4a$.

In order to achieve the desired bound, we need to exploit  the cancellation arising from the character sum over $h$. To that end, we apply the  P\'{o}lya-Vinogradov inequality.

\begin{lemma}[{\cite[Theorem 9.18]{MontgometryVaughanbook}}]\label{l3}
Let $\chi$ be an non-principal Dirichlet character modulo $a$. We have 
\[
\left| \sum_{M < n \le M + N} \chi(n) \right| \le \sqrt{a} \log a. 
\]
\end{lemma}
Next, we prove the following lemma on the partial sum $S(N)$. 
\begin{lemma}\label{l4} 
We have 
$$
S(N)=\sum_{\substack{h_1\le N\\(h_1,a_1)=1}}S(h_1,a_1) \ll
\left\{
\begin{array}{ll}
N\sqrt{a_1}, &\text{if } a_1 \text{ is a square},\\
a_1\log a_1, &\text{otherwise}.
\end{array}
\right.
$$
\end{lemma}
\begin{proof}

If $a_1$ is a square, then by Lemma~\ref{l2}, we have
$$
S(N) \ll\sum_{h_1\le N}\sqrt{a_1}=N\sqrt{a_1}.
$$

Next we treat the case when $a_1$ is not a square. It is readily verified that 

$$
\varepsilon_m^{-1}=\frac{1-i}2\chi_{{}_{\scriptstyle0}}(m)+\frac{1+i}2\chi_{{}_{\scriptstyle1}}(m), 
$$
where $\chi_{{}_{\scriptstyle0}}$ is the principal character modulo 4, and $\chi_{{}_{\scriptstyle1}}$ is the quadratic character modulo 4, i.e.
$$
\chi_{{}_{\scriptstyle1}}(n)=\begin{cases}
0,& \text{if } 2|n, \\
1,&\text{if } n \equiv1\pmod{4}, \\
-1,&\text{if } n\equiv3\pmod{4}.
\end{cases}
$$

If  $a_1$ is odd, let $\chi=\legendre{*}{a_1}$; if $a_1$ is even, let $\chi=\chi_{{}_{\scriptstyle0}}\legendre{a_1}{*}$ or $\chi=\chi_{{}_{\scriptstyle1}}\legendre{a_1}{*}$. Note that in any case, $\chi$ is always a non-principal character of modulus at most $4a_1$. Therefore, we may apply the P\'{o}lya-Vinogradov inequality (Lemma~\ref{l3}) to the character sum $ \sum_{h_1\le N}\chi(h_1)$ after another application of Lemma~\ref{l2}, and obtain 
$$
S(N)\ll a_1\log a_1.
$$
\end{proof}
By Lemma~\ref{l4} and partial summation, we obtain
$$
\sum_{\substack{h_1\le H/d\\(h_1,a_1)=1}}\frac{H-h_1d}{H^2}S(h_1,a_1)
\ll
\left\{
\begin{array}{ll}
\frac{\sqrt{a_1}}{d}, &\text{if } a_1 \text{ is a square},\\
\frac{a_1\log a_1}{H}, &\text{otherwise}.
\end{array}
\right.
$$
Therefore, recalling that $r$ is the largest integer such that $r^2 \mid a$
and $H=\left\lfloor \frac{1}{2\delta}\right\rfloor$, we have 
\begin{align*}
E(a)\ll &\sum_{\substack{d|a\\a_1 = \frac{a}{d}=\square}}
d \frac{\sqrt{a_1}}{d}
+\sum_{\substack{d|a\\a_1 = \frac{a}{d} \neq \square}}d \frac{a_1 \log a_1}{H}\\
\ll & 
\sum_{\substack{a_1|a\\a_1=\square}} \sqrt{a_1}
+\frac{a}{H}\sum_{\substack{a_1|a\\a_1\neq\square}}\log a\\
\ll&\sigma(r)+\delta a(\log a)d(a).
\end{align*}

This completes the proof of Theorem~\ref{t3}. 
Theorem~\ref{t4} follows on using the same argument as that used to prove Theorem~\ref{t3}. The only difference is that we use the following statement in place of Lemma~\ref{l3}. 
\begin{lemma}[{\cite[Theorem 2]{MontgometryVaughan1977}}]\label{l5}
Let $\chi$ be an non-principal Dirichlet character modulo $a$. Suppose that the Generalized Riemann Hypothesis holds true, then 
\[
\left| \sum_{M < n \le M + N} \chi(n) \right| \le \sqrt{a} \log\log 3a. 
\]
\end{lemma}

\proof[Acknowledgement]
The authors are grateful to the anonymous referee for helpful suggestions and comments. 

\bibliographystyle{plain}
 \bibliography{biblio}

\end{document}